\pdfoutput=1
\RequirePackage{ifpdf}
\ifpdf 
\documentclass[pdftex]{sigma}
\else
\documentclass{sigma}
\fi

\numberwithin{equation}{section}

\newtheorem{Theorem}{Theorem}[section]
\newtheorem{Corollary}[Theorem]{Corollary}
\newtheorem{Lemma}[Theorem]{Lemma}
 { \theoremstyle{definition}
\newtheorem{Definition}[Theorem]{Definition}
\newtheorem{Example}[Theorem]{Example}
\newtheorem{Remark}[Theorem]{Remark} }

\newcommand{\dee}{\partial}
\newcommand{\K}{K[\dee_t,\dee_{x_1}, \dots, \dee_{x_n}]}
\newcommand{\diffs}{\dee_t,\dee_{x_1}, \dots, \dee_{x_n}}
\newcommand{\Ko}{K}
\newcommand{\ci}{c_{p_1 p_2 \dots p_n}}
\newcommand{\vars}{x_1, \dots, x_n}

\begin{document}

\allowdisplaybreaks

\newcommand{\arXivNumber}{1605.04362}

\renewcommand{\PaperNumber}{010}

\FirstPageHeading

\ShortArticleName{Classif\/ication of Multidimensional Darboux Transformations}

\ArticleName{Classif\/ication of Multidimensional Darboux\\ Transformations: First Order and Continued Type}

\Author{David HOBBY and Ekaterina SHEMYAKOVA}

\AuthorNameForHeading{D.~Hobby and E.~Shemyakova}

\Address{1 Hawk dr., Department of Mathematics, State University of New York at New Paltz, USA}
\Email{\href{mailto:hobbyd@newpaltz.edu}{hobbyd@newpaltz.edu}, \href{mailto:shemyake@newpaltz.edu}{shemyake@newpaltz.edu}}

\ArticleDates{Received October 10, 2016, in f\/inal form February 14, 2017; Published online February 24, 2017}

\Abstract{We analyze Darboux transformations in very general settings for multidimensional linear partial dif\/ferential operators. We consider all known types of Darboux transformations, and present a new type. We obtain a full classif\/ication of all operators that admit Wronskian type Darboux transformations of f\/irst order and a complete description of all possible f\/irst-order Darboux transformations. We introduce a large class of invertible Darboux transformations of higher order, which we call Darboux transformations of continued Type~I. This generalizes the class of Darboux transformations of Type~I, which was previously introduced. There is also a modif\/ication of this type of Darboux transformations, continued Wronskian type, which generalize Wronskian type Darboux transformations.}


\Classification{16S32; 37K35; 37K25}

\section{Introduction}

A Darboux transformation (DT), in the general sense of the word, is a transformation between dif\/ferential operators that simultaneously transforms their kernels (solution spaces) or eigenspaces. DTs originated in the work of Darboux and others on the theory of surfaces, as in~\cite{Darboux2}, while particular examples were known to Euler and Laplace. They were rediscovered in integrable systems theory in the 1970s, where they were used for obtaining solutions of soliton-type equations with remarkable properties. They have also been studied by physicists for quantum-mechanical applications~\cite{infeld_hull_factorization_method,shroedinger_on_DT}. The very name `Darboux transformations' was introduced by V.B.~Matveev in~\cite{Matveev79}, where he played a central role in creating the Darboux transformation method in soliton theory. This theory was elaborated in the fundamental monograph~\cite{matveev:1991:darboux} of V.B.~Matveev and M.A.~Salle (see also~\cite{doktorovleble2007dressing}). As has been pointed out by experts in the f\/ield such as Novikov, Darboux transformations play a more important role than merely as technical tools for constructing solutions.

While a large amount of work has been devoted to Darboux transformations, it has concentrated mainly on particular examples and specif\/ic operators, while a general theory was largely missing. Although starting from Darboux's own work, there have been ideas, observations, and conjectures that can be interpreted as elements of this sought-after general theory of Darboux transformations. In view of the numerous applications of Darboux transformations, the importance of developing such a theory cannot be overestimated. In recent years, important steps in this direction have been made. They are based on the algebraic approach in the center of which are so-called intertwining relations.

The model example of Darboux transformations is the following transformation of single variable Sturm--Liouville operators: $L\to L_1$, where $L=\dee_x^2+u(x)$, $L_1=\dee_x^2+u_1(x)$, and $u_1(x)$ is obtained from $u(x)$ by the formula $u_1(x)=u(x) + 2(\ln\varphi_0(x))_{xx}$. Here $\varphi_0(x)$ is a `seed' solution of the Sturm--Liouville equation $L\varphi_0=\lambda_0\varphi_0$ (with some f\/ixed $\lambda_0$).
Then the transformation $\varphi\mapsto (\dee_x-(\ln\varphi_0)_x)\varphi$ sends solutions of $L\varphi=\lambda\varphi$ to solutions of $L_1\psi=\lambda\psi$ (with the same $\lambda$). The seed solution is mapped to zero. (According to~\cite{novikov-dynnikov:1997}, this transformation was already known to Euler.) This example is a model in two ways.

First, the formula for the transformation of solutions can be written in terms of Wronskian determinants: $\varphi_1=W(\varphi_0,\varphi)/\varphi_0$. This generalizes to a construction based on several linearly independent seed solutions and higher-order Wronskian determinants (Crum~\cite{Crum:1955:1999} for Sturm--Liouville operators and Matveev~\cite{Matveev79} for general operators on the line).

Second, if one lets $M=\dee_x-(\ln\varphi_0)_x$, then the following identity is satisf\/ied:
\begin{gather}\label{intertw1}
ML=L_1M.
\end{gather}
This identity is equivalent to the relation between the old potential $u(x)$ and the new poten\-tial~$u_1(x)$.

In an abstract framework, if two operators with the same principal symbol, $L$ and $L_1$, satisfy~\eqref{intertw1} for some~$M$, then~\eqref{intertw1} is called the \emph{intertwining relation} and $M$ is (often) called the \emph{transformation operator}. One can see that if~\eqref{intertw1} is satisf\/ied, then the operator $M$ def\/ines a~linear transformation of the eigenspaces of $L$ to the eigenspaces of $L_1$ (for all eigenvalues). The relation~\eqref{intertw1} can be taken as a def\/inition of the DT. The intertwining relation~\eqref{intertw1} appeared, for Sturm--Liouville operators, in work of Shabat~\cite[equation~(19)]{shabat:1992}, Veselov--Shabat~\cite{veselov-shabat:dress1993}, and Bagrov--Samsonov~\cite{bagrov-samsonov:factorization1995}.
Intertwining relation~\eqref{intertw1} is also related with supersymmetric quantum mechanics initiated by E.~Witten~\cite{witten1982}, see in particular~\cite{Ioffe_Junker99, Cooper_Khare95}. In 2D case intertwining relation~\eqref{intertw1} was studied in the series of papers by
A.~Andrianov, F.~Cannata, M.~Iof\/fe, see, e.g.,~\cite{ioffe2010separation_of_vars} and references therein. It also appeared for higher dimensions, in Berest--Veselov~\cite{berest-veselov:1998, berest-veselov:2000} for the Laplace type operators $L=-\Delta +u$. The same relation was used in~\cite{2015:super} for dif\/ferential operators on the superline. The natural task that arises for such an algebraic def\/inition of DTs, is the classif\/ication of the DTs satisfying it. In the case of dif\/ferential operators on the line, it was established (in steps) that all DTs def\/ined this way arise from seed solutions and are given by Wronskian formulas. For the Sturm--Liouville operators, this was proved in~\cite[Theorem 5]{veselov-shabat:dress1993} when the new potential $u_1(x)$ dif\/fers from the initial $u(x)$ by a constant, $u_1(x)=u(x) + c$. It was proved in~\cite{shabat1995} for transformations of order two; and, f\/inally in the general case, in~\cite{bagrov-samsonov:factorization1995} and the follow-up paper~\cite{clouds:bagrov:samsonov:97}, see also~\cite[Section~3]{samsonov:factorization1999}. For general operators on the line, it was proved in~\cite{adler-marikhin-shabat:2001}. For the superline, the classif\/ication was obtained in~\cite{2015:super}.

The intertwining relation~\eqref{intertw1}, with a single transformation operator $M$, turns out to already be too restrictive for dif\/ferential operators in higher dimensions. We use the more f\/lexible intertwining relation
\begin{gather}\label{intertw2}
NL=L_1M.
\end{gather}
This can be extracted from the work of Darboux himself~\cite{Darboux2} (see for example~\cite[equation~(2)]{TsarevS2009}). There are several important dif\/ferences between the two kinds of intertwining relations. If~\eqref{intertw1} is satisf\/ied, then, as already mentioned, the operator $M$ transforms each eigenspace of $L$ to the eigenspace of~$L_1$ with the same eigenvalue. In contrast with that, if~\eqref{intertw2} is satisf\/ied, then we only have that~$M$ maps the kernel of $L$ to the kernel of~$L_1$. (It is in general false that $M$ maps eigenspaces of $L$ with nonzero eigenvalues to eigenspaces of~$L_1$.) In~\cite{shemyakova2013_DT_fact, shem:darboux2}, a general framework for Darboux transformations def\/ined by intertwining relation~\eqref{intertw2} was put forward, which in particular allowed the proof of a~long-standing conjecture of Darboux on factorization of the DTs for 2D second-order hyperbolic operators (the `2D Schr\"odinger operator')\footnote{Calling it Schr\"odinger is some abuse of language justif\/ied by the fact that on the formal algebraic level it is equivalent to the actual elliptic Schr\"odinger operator~\cite{novikov-dynnikov:1997}.}. As we explain in greater detail in the paper, it is natural to introduce an equivalence relation between pairs $(M,N)$ in conjunction with the intertwining relation~\eqref{intertw2}. This gives extra f\/lexibility and (as explained below, see Section~\ref{sec:Known_types_of_Darboux transformations}) makes it possible to have invertible DTs. Such an equivalence relation does not exist in the case of the intertwining relation~\eqref{intertw1} with $M=N$. It remains an open problem to analyze fully the connections between the intertwining relations~\eqref{intertw1} and~\eqref{intertw2}. In~\cite{2015:inv:charts}, a completely new class of transformations, Darboux transformations of Type~I, was described. These are for operators of very general form, and are analogous to Laplace transformations. In~\cite{2015:super}, a complete classif\/ication of Darboux transformations of arbitrary nondegenerate operators on the superline was obtained. In~\cite{shemya:voronov2016:berezinians} it was proved that every Darboux transformation can be obtained in terms of super-Wronskians, which are based on Berezinians, analogues of determinants in the super case.

Among DTs def\/ined by our intertwining relation~\eqref{intertw2} there are non-Wronskian type Darboux transformations such as the classical Laplace transformations for the hyperbolic second-order operator in the plane, sometimes called the 2D Schr\"odinger operator. In~\cite{shemyakova2013_DT_fact,shem:darboux2} it was proved that every Darboux transformation for the class of operators is a composition of atomic Darboux transformations of two types: Wronskian type and Laplace transformations (a conjecture that can be traced back to Darboux). Since~\eqref{intertw2} includes \eqref{intertw1} as a special case, we will henceforth use {\em intertwining relation} to refer to~\eqref{intertw2}.

In this paper we recall the general algebraic framework for Darboux transformations based on intertwining relations and then analyze the various types of Darboux transformations that may arise. We partially answer the following questions in the general theory of Darboux transformations:
\begin{itemize}\itemsep=0pt
 \item \textit{Which classes of operators admit Darboux transformations of Wronskian type?}
 \item \textit{Are there any new types of invertible Darboux transformations?}
 \item \textit{What is the qualitative difference between these two types?}
\end{itemize}

For more than a century, only two types of transformations of linear partial dif\/ferential operators satisfying our intertwining relation (Darboux transformations) were known: the transformations def\/ined by Wronskian formulas, and Laplace transformations. Note that Wronskian formulas def\/ine DTs for several dif\/ferent kinds of operators, see~\cite{Matveev79,matveev:1991:darboux}, while Laplace transformations are def\/ined only for 2D Schr\"odinger operators. Laplace transformations have a number of good properties including invertibility. Recently, two dif\/ferent generalizations of Laplace transformations were proposed: ``intertwining Laplace transformations'' in~\cite{ganzha2013intertwining} and ``Type~I transformations''. (The latter are always invertible.) In the present paper we clarify their relation and introduce a further generalization, which we call transformations of ``continued Type~I''. (This term comes from the construction, which has some resemblance with that of continued fractions.) A similar construction can also be applied starting from f\/irst-order DTs of Wronskian type, we say the resulting DTs are of ``continued Wronskian type''.

We also obtain a full description of f\/irst-order DTs and the dif\/ferential operators that admit them\footnote{Actually, we consider transformations with $M= \partial_t +m$, where $m \in K$ and $t$ is a distinguished variable. It should be possible to reduce a general f\/irst-order operator $M$ to this form by a change of variables.}. Moreover, we show that f\/irst-order DTs always fall into two types: Wronskian type and Type~I.
These two dif\/ferent cases possess a unif\/ied algebraic presentation (see Theorem~\ref{thm:commutator}). Our test makes it possible to tell which type of DTs are admitted by a given operator.

The structure of the paper is as follows. In Section~\ref{sec:Known_types_of_Darboux transformations}, we recall key results on the category of Darboux transformations as introduced in~\cite{shemyakova2013_DT_fact,shem:darboux2} and list known types of Darboux transformations. In Section~\ref{sec:wronski}, we give a criterion for a general operator to admit a f\/irst-order Wronskian Darboux transformation (Theorem~\ref{thm:wronski_type_char}). In Section~\ref{sec:1storder_classification} we classify f\/irst-order DTs. In Section~\ref{sec:it}, we introduce continued Type~I Darboux transformations and prove they are always invertible. In Section~\ref{sec:weak:wr}, we give a similar construction of DTs of continued Wronskian type.

\section{Types of Darboux transformations}\label{sec:Known_types_of_Darboux transformations}

In this section we recall some general facts about Darboux transformations. Consider a dif\/fe\-ren\-tial f\/ield $\Ko$ of characteristic zero with commuting deriva\-tions~$\diffs$. We use one distinguished letter~$t$ to denote the variable which will play a special role. The letter~$t$ need not represent time, and was chosen for convenience. By~$\K$, we denote the corresponding ring of linear partial dif\/ferential operators over~$K$. Operators in $\K$ of order~$0$ will often be considered as elements of $K$, and called ``functions''. One can either assume the f\/ield~$K$ to be dif\/ferentially closed, or simply assume that~$K$ contains the solutions of those partial dif\/ferential equations that we encounter on the way.

Darboux transformations viewed as mappings of linear partial dif\/ferential operators can be def\/ined algebraically as follows~\cite{shem:darboux2},
where we write $\sigma(L)$ for the principal symbol of~$L$.

\begin{Definition} \label{def:darb}
Consider a category, with objects all operators in $\K$ and morphisms the \textit{Darboux transformations} def\/ined as follows.
A morphism from an object $L$ to an object $L_1$ with $\sigma(L) = \sigma(L_1)$ is a pair $({M},{N})\in \K \times \K$ satisfying the \textit{intertwining relation}
\begin{gather}\label{intertwining_relation}
{{N} {L} = {L}_1 {M}}.
\end{gather}

Composition of Darboux transformations is def\/ined by $(M,N)\circ(M_1,N_1) =(M_1 M ,N_1 N)$, where the left DT is applied f\/irst. For each object $L$, its \textit{identity morphism} $1_L$ is $(1,1)$.
\end{Definition}

Observe that given $NL = L_1M$, we have $\sigma(L) = \sigma(L_1)$ if\/f $\sigma(N) = \sigma(M)$. Note also that since there are only morphisms between objects with the same principal symbol, this category is partitioned into subcategories for each principal symbol.

\begin{Definition}\label{equivalence relation definition}
Two morphisms $(M,N)$ and $(M',N')$ from $L$ to $L_1$ are {\em equivalent}, written $(M,N) \sim (M',N')$, if\/f there exists $A \in \K$ with $M' = M + AL$ and $N' = N + L_1 A$.
\end{Definition}

\begin{Lemma} The relation $\sim$ in Definition~{\rm \ref{equivalence relation definition}} is an equivalence relation. When $(M,N)$ is a~morphism from~$L$ to $L_1$, so is every pair of the form $({M} +{A} {L} , {N} + {L}_1 {A})$. The relation $\sim$ is compatible with the composition.
\end{Lemma}
\begin{proof} By construction, $\sim$ is clearly an equivalence relation. To see the next statement, suppose a pair~$(M,N)$ satisf\/ies~\eqref{intertwining_relation}, and consider an equivalent pair $(M+AL,N+L_1A)$. We have $(N+L_1A)L=NL+L_1AL=L_1M + L_1AL=L_1(M+AL)$, so it also satisf\/ies the intertwining relation.

Finally, we show that $\sim$ is compatible with composition. Consider the Darboux transformations $(M,N)\colon L \rightarrow L_1$, and $(M_1,N_1)\colon L_1 \rightarrow L_2$, where all operators are in~$\K$. Taking dif\/ferent representatives in the same equivalence classes, we have $(M+AL,N+L_1A)\colon$ $L \rightarrow L_1$, and $(M_1+BL_1,N_1+L_2B)\colon L_1 \rightarrow L_2$ for some ${A}, {B} \in \K$. That is we have: $({N} + {L}_1 {A}) {L} = {L}_1 ({M} + {A} {L})$, and $({N}_1 + {L}_2 {B}) {L}_1 = {L}_2 ({M}_1 + {B} {L}_1)$. For the composition we have $(\widehat{M}, \widehat{N} )\colon L \rightarrow L_2$ with $\widehat{N}=({N}_1 + {L}_2 {B}) ({N} + {L}_1{A})$, $\widehat{M} = ({M}_1 + {B} {L}_1) ({M} + {A} {L})$. Using $NL=L_1M$ and ${N}_1 {L}_1=L_2 M_1$ we have that $\widehat{N}={N}_1 {N} + {L}_2 {C}$, and $\widehat{M} = {M}_1 {M} + C L$, where ${C}={B} {N}+B L_1{A} + M_1 A$.
This concludes the proof.
\end{proof}

Every Darboux transformation $(M,N)$ in a given equivalence class def\/ines the same linear mapping from $\ker L$ to $\ker L_1$, where each function~$\phi$ goes to $M[\phi]$. For if $\phi \in \ker L$ then $(M+AL)[\phi] = M[\phi] + A[L[\phi]] = M[\phi]+A[0] = M[\phi]$, so $M$ and $M+AL$ give the same linear map. This motivates our use of the equivalence relation $\sim$, as in Def\/inition~\ref{equivalence relation definition}.

Note that one may also view the $\sim$ classes as being the Darboux transformations, rather than the individual pairs of operators. This approach was taken in
\cite{shemyakova2013_DT_fact,2013:invertible:darboux, shem:darboux2} and \cite{2015:inv:charts}.

We def\/ine the {\em order} of a Darboux transformation $(M,N)$ as the minimum possible order of a~transformation in its equivalence class, that is the least possible order of an operator of the form $M+AL$. We also def\/ine invertibility in terms of equivalence classes, and say that the Darboux transformation $(M,N)\colon L \rightarrow L_1$ is {\em invertible} if\/f there exists $(M',N')\colon L_1 \rightarrow L$ such that the compositions $(M,N) \circ (M',N')$ and $(M',N') \circ (M,N)$ are equivalent to identity morphisms. Since equivalent morphisms give the same kernel maps, this is justif\/ied. Continuing this abuse of terminology, we say that the morphisms $(M,N)$ and $(M',N')$ are {\em inverses}, and so on. Note that with this def\/inition, a morphism may have multiple inverses, which will be equivalent but not equal.

This def\/inition gives us that $(M,N) \colon L \rightarrow L_1$ and $(M',N') \colon L_1 \rightarrow L$ are inverses if\/f there are some $A,B \in \K$ with
\begin{gather}
{{M'M=1+AL}} , \label{inv:prop1} \\
N'N=1+LA , \label{inv:nl} \\
{{MM'=1+BL_1}} , \label{inv:prop2} \\
NN'=1+L_1B . \label{inv:nl1}
\end{gather}

\begin{Lemma} Invertible Darboux transformations induce isomorphisms on the kernels of the operators~$L$ and~$L_1$.
\end{Lemma}
\begin{proof}
Consider $\psi \in \ker L$, then from~\eqref{inv:prop1} we obtain that $M'M[\psi]=\psi$. Similarly for $\psi_1 \in \ker L_1$, \eqref{inv:prop2} gives us $MM'[\psi_1]=\psi_1$. Therefore, $M$ and~$M'$ induce mutually invertible maps between~$\ker L$ and~$\ker L_1$.
\end{proof}

In particular, \eqref{inv:prop1} implies that $\ker {L} \cap \ker {M}=\{0\}$ is necessary for a Darboux transformation to be invertible. Also note that the order of a Darboux transformation is not related in an obvious way to the order of its inverse.

The f\/irst steps in the general study of invertible Darboux transformations along with the analysis of one particular invertible class of Darboux transformations~-- Darboux transformations of Type~I~-- can be found in~\cite{2013:invertible:darboux,2015:inv:charts}.

\begin{Lemma} \label{lem:gn=ma}\quad
\begin{enumerate}\itemsep=0pt
 \item[$1.$] Given the intertwining relations $NL = L_1 M$ and $N' L_1 = LM'$, the equalities for~$N$,~$N'$ follow from the equalities for $M$, $M'$. Specifically, \eqref{inv:nl}~follows from~\eqref{inv:prop1}, and \eqref{inv:nl1} follows from \eqref{inv:prop2},
 \item[$2.$] If $(M,N)$ and $(M',N')$ are mutually inverse morphisms satisfying~\eqref{inv:prop1} through~\eqref{inv:nl1}, then $BN=MA$.
\end{enumerate}
\end{Lemma}
\begin{proof} 1. Assume that \eqref{inv:prop1} holds, so $M'M=1+AL$. Then $N'NL=LM'M=L(1+AL)=(1+LA)L$. Since $\K$ has no zero divisors, we get $N'N = 1+LA$ by cancellation, so \eqref{inv:nl} holds. That \eqref{inv:prop2} implies \eqref{inv:nl1} is proved analogously.

2. We know that $NN'=1+L_1B$. Then $NN'N=N+L_1BN$, which implies $N(1+LA)=N+L_1BN$. Thus, $NLA=L_1BN$, which implies $L_1MA=L_1BN$ using the intertwining relation $NL=L_1M$. Therefore, $BN=MA$ by cancellation.
\end{proof}

Recall that a \textit{gauge transformation} of an operator $L$ is def\/ined as $L^{g}= g^{-1} L g$, where $g \in \Ko$. Gauge transformations commute with Darboux transformations in the sense that if there is a~Darboux transformation $(M,N)\colon L \rightarrow L_1$, then there are also Darboux transformations
\begin{gather*}
\big(M^g,N^g\big)\colon \ L^g \rightarrow L_1^g , \qquad
(Mg,Ng)\colon \  L^g \rightarrow L_1 , \qquad
\big(g^{-1}M,g^{-1}N\big)\colon \ L \rightarrow L_1^g .
\end{gather*}

\begin{Definition}\label{shift definition}
If a pair $(M,N)$ def\/ines a Darboux transformation from $L$ to $L_1$, then the same pair def\/ines a dif\/ferent Darboux transformation from $L+AM$ to $L_1+NA$ for every $A \in \K$. We shall say that Darboux transformations $L \rightarrow L_1$ and $L+AM \rightarrow L_1+NA$ are related by a {\em shift}. Note that this is a symmetric relation.
\end{Definition}

The shift of the DT $(M,N) \colon L \rightarrow L_1$ is a DT. We have that $N(L+AM) = (L_1 + NA)M$ for the intertwining relation, and $\sigma(L) = \sigma(L_1)$ implies $\sigma(M) = \sigma(N)$ which implies $\sigma(L+AM) = \sigma(L_1 + NA)$.

In general, shifts of Darboux transformations do not commute with compositions of Darboux transformations. However, shifts are useful when dealing with invertible Darboux transformations.

\begin{Lemma}\label{shift invertibility lemma}
Shifts of Darboux transformations preserve invertibility.
\end{Lemma}
\begin{proof} Suppose we have a Darboux transformation $(M,N)\colon L \rightarrow L_1$ with $NL=L_1M$, which is invertible, and there is a Darboux transformation $(M',N')\colon L' \rightarrow L$ with $N'L_1=LM'$, and~\eqref{inv:prop1} and~\eqref{inv:prop2} hold, giving $M'M = 1+AL$ and $MM' = 1+BL_1$. Then for any $C \in \K$ we have the Darboux transformation $(M,N)\colon \widetilde{L} \rightarrow \widetilde{L_1}$, where $\widetilde{L}=L+CM$, and $\widetilde{L_1}=L_1+NC$. The inverse Darboux transformation for the shifted Darboux transformation is then given by $\widetilde{M'}=M'+AC$ and $\widetilde{N'}=N'+CB$. We f\/irst show $\widetilde{N'}\widetilde{L_1} = \widetilde{L}\widetilde{M'}$, or $(N'+CB)(L_1+NC)=(L+CM)(M'+AC)$. This is equivalent to $N'L_1+N'NC+CBL_1+CBNC=LM'+CMM'+CMAC+LAC$. Taking into account that $N'L_1=LM'$ and $BN=MA$, the last equality is equivalent to $N'NC+CBL_1=CMM'+LAC$, which is equivalent to $(1+LA)C+CBL_1=C(1+BL_1)+LAC$, which is true.

Let us now prove that $\widetilde{M'}M=1+A \widetilde{L}$. This equality is equivalent to $(M'+AC)M=1+A(L+CM)$, which is true since $M'M=1+AL$.

It remains to show $M\widetilde{M'}=1+B \widetilde{L_1}$, or $M(M'+AC)=1+B(L_1+NC)$. Since $MA=BN$ by Lemma~\ref{lem:gn=ma}, this is equivalent to $MM'=1+BL_1$, which is true. The equalities for~$N$ and~$N'$ follow from Lemma~\ref{lem:gn=ma} as well.
\end{proof}

Another useful idea is what we call the {\em dual} of a Darboux transformation.

\begin{Definition}\label{dual definition}
If $(M,N) \colon L \rightarrow L_1$ is a Darboux transformation, its {\em dual} is the Darboux transformation $(L,L_1) \colon M \rightarrow N$.
\end{Definition}

This def\/ines a Darboux transformation. If $(M,N) \colon L \rightarrow L_1$ is a DT, then $NL = L_1 M$ and $\sigma(L) = \sigma(L_1)$. Thus $L_1 M = NL$, and $\sigma(M) = \sigma(N)$, and $(L,L_1) \colon M \rightarrow N$ is also a DT.

Note that this is not a duality in the sense of category theory; our dual is quite unusual in the sense that it switches objects and morphisms.

This duality gives a useful perspective on the category of DTs. For example, we have that DTs are related by a shift if and only if their duals are equivalent. To see this, consider the morphism $(M,N)$ from $L$ to $L_1$, and its shift where $(M,N)$ is a~morphism from $L+AM$ to $L_1+NA$. Taking duals gives us two morphisms from $M$ to $N$, one given by the pair $(L,L_1)$ and the other by the pair $(L+AM,L_1+NA)$, which the reader may verify are equivalent. The proof
in the other direction is similar. More importantly, we have the following lemma.

\begin{Lemma}\label{dual is invertible lemma}
A Darboux transformation is invertible if and only if its dual is.
\end{Lemma}

\begin{proof}
Since the dual of the dual of a DT is the original transformation, it suf\/f\/ices to show that the dual of an invertible DT is invertible. So let $(M,N) \colon L \rightarrow L_1$ be an invertible DT. This gives us a DT $(M',N') \colon L_1 \rightarrow L$ with $M'M = 1 + AL$ and $M M' = 1 + B L_1$ for some
operators $A$ and $B$. Lemma~\ref{lem:gn=ma} gives us $N' N = 1 + LA$ and $BN = MA$. Now we must show that $(L,L_1) \colon M \rightarrow N$
is invertible. That is, we must show that $(\widetilde{M},\widetilde{N}) \colon \widetilde{L} \rightarrow \widetilde{L_1}$ is invertible, where
we have $\widetilde{M} = L$, $\widetilde{N} = L_1$, $\widetilde{L} = M$ and $\widetilde{L}_1 = N$. So we need a DT $(\widetilde{M'},\widetilde{N'}) \colon \widetilde{L_1} \rightarrow \widetilde{L}$, with $\widetilde{M'} \widetilde{M} = 1 + \widetilde{A} \widetilde{L}$ and $\widetilde{M} \widetilde{M'} = 1 + \widetilde{B} \widetilde{L_1}$ for some operators $\widetilde{A}$ and $\widetilde{B}$.

Now we take $\widetilde{M'} = -A$, $\widetilde{N'} = -B$, $\widetilde{A} = -M'$ and $\widetilde{B} = -N'$. We have $\widetilde{N'} \widetilde{L_1} = -B N = M (-A) = \widetilde{L} \widetilde{M'}$, so $(\widetilde{M'},\widetilde{N'}) \colon \widetilde{L_1} \rightarrow \widetilde{L}$ is a DT. Also
$\widetilde{M'} \widetilde{M} = -A L = 1 + (-M')M = 1 + \widetilde{A} \widetilde{L}$, and $\widetilde{M} \widetilde{M'} = L (-A) = 1 + (-N')N = 1 + \widetilde{B} \widetilde{L_1}$, as required.
\end{proof}

There are several known types of Darboux transformations.

{\bf 1.~Darboux transformations of Wronskian type.} These are the classical DTs. The opera\-tor~$M$ is given by
\begin{gather*}
M(f) = \frac{W(f_1, f_2, \dots , f_m,f)}{W(f_1, f_2, \dots , f_m)},
\end{gather*}
where $W(f_1, f_2, \dots , f_m)$ denotes a Wronskian determinant with respect to one of the variables $t$,~$\vars$ of $m$ linearly independent $f_j \in \Ko$, which are elements of $\ker L$.

\begin{Definition}\label{Wronskian definition}
We say that a DT $(M,N) \colon L \rightarrow L_1$ is of {\em Wronskian type} when $M$ can be put in the above form by a change of variables, and we say that the DT is a {\em multiple of Wronskian type} if $M$ is the result of multiplying an operator of the above form on the left by a function in~$K$.
\end{Definition}

Wronskian type Darboux transformations are proved to be admitted by several dif\/ferent types of  operators~\cite{2015:super,nimmo2010, Matveev79,matveev:1991:darboux}, and direct applications of these transformations to solve famous nonlinear equations are well known. Wronskian type transformations are never invertible, since $\ker L \cap \ker M \neq \{0\}$. (Wronskian type DT formulas in abstract frameworks were introduced in~\cite{etingof-gelfand-retakh:1997} and~\cite{nimmo2010}. An analog of Wronskian DTs for super Sturm--Liouvlle operators was discovered in~\cite{liu95_initial,liu_manas97,liu_manas97_2}.)

Note that in the 1D case with intertwining relation having $M=N$, as in \eqref{intertw1}, a similar Wronskian construction is used (it has the same formula), but the~$f_i$ are arbitrary eigenfunctions with not necessarily zero eigenvalues. For the 1D case, every DT is of Wronskian type in this extended sense,
see~\cite{adler-marikhin-shabat:2001}.

{\bf 2.~Darboux transformations obtained from a factorization.} Suppose $L=CM$ for some $C, M \in \K \setminus K$. Then for any operator $N$ with
$\sigma(N) = \sigma(M)$, there is a~DT
\begin{gather*}
(M,N) \colon \ CM \rightarrow NC ,
\end{gather*} since $N(CM) = (NC)M$. Since $\ker M \subseteq \ker L$, these transformations are never invertible.

A common trick in proofs is to establish using some reasoning that the transformation operator $M$ can be considered in a form where it is ef\/fectively an ordinary dif\/ferential operator while everything else is multidimensional. In this case, if the Darboux transformation is obtained from a factorization with this particular $M$, then it is also of Wronskian type. This uses the fact that every linear ordinary dif\/ferential operator can be expressed by a Wronskian formula.

{\bf 3.~Laplace transformations.} These are another type of Darboux transformations, introduced in~\cite{Darboux2}. They are distinct from the Wronskian type. They are only def\/ined for 2D Schr\"odinger type operators, which have the form
\begin{gather} \label{eq:L}
L = \partial_x\partial_y +a \partial_x + b\partial_y +c ,
\end{gather}
where $a, b, c \in K$. If the {\em Laplace invariant} $k = b_y + ab -c$ is nonzero, then $L$ admits a Darboux transformation with $M = \partial_x + b$ (in ``$x$-direction''). (Explicit formulas for $L_1$ and $N$ are given below.) If the other Laplace invariant, $h = a_x + ab -c$ is nonzero, then $L$ admits a~Darboux transformation with $M = \partial_y + a$ (in ``$y$-direction'').

Laplace transformations are invertible, and are (almost) inverses of each other. This has been mentioned in the literature, e.g., \cite{ts:steklov_etc:00}. Classically, the invertibility of Laplace transformations was understood to mean that they induced isomorphisms of the kernels of the operators in question. Interestingly, it is exactly the equivalence relation on Darboux transformations that makes it possible to understand the invertibility of Laplace transformations in the precise algebraic sense. The following statement appeared in a brief form in~\cite{2013:invertible:darboux}.

\begin{Theorem}\quad
\begin{enumerate}\itemsep=0pt
 \item[$1.$] The composition of two consecutive Laplace transformations applied to $L$, first in $x$ direction, and then in~$y$ direction is equal to the gauge transformation $L \rightarrow L^{1/k}$. If the transformation is first in $x$ direction, and then in~$y$ direction then the composition is equal to the gauge transformation $L \rightarrow L^{1/h}$.
 \item[$2.$]
The inverse for the Laplace transformation $L \rightarrow L_1$ given by the operator $M=\partial_x + b$ is $(M',N')\colon L_1 \rightarrow L$, where
$M'=-k^{-1} ( \partial_y + a)$, $N'=-k^{-1} ( \partial_y + a- k_y k^{-1})$. The inverse for the Laplace transformation $L \rightarrow L_1$ given by the operator $M=\partial_y + a$ is $(M',N')\colon L_1 \rightarrow L$, where $M'=-h^{-1} ( \partial_x + b)$, $N'=-h^{-1} ( \partial_x + b- h_y h^{-1} )$.
\end{enumerate}
\end{Theorem}

\begin{proof}
Consider the Laplace transformation of $L$ of the form~\eqref{eq:L} in ``$x$-direction'', that is given by the operator $M=\partial_x + b$. The operators $M$ and $L$ in this case completely def\/ine opera\-tors~$L_1$ and~$N$. (This is an instance of our Theorem~\ref{L and M determine theorem}.) Note that expressions for $L_1$ and $N$ are much shorter, when expressed in terms of $k$ (substitute $c = b_y + ab -k$): $N =\partial_x + b-k_x k^{-1}$, $L_1 =\partial_x \partial_y+a_1 \partial_x + b_1 \partial_y+ c_1 =\partial_x \partial_y+a \partial_x + (b - k_x k^{-1}) \partial_y+ ab - a k_x k^{-1} - k - a_x$.

Now consider the Laplace transformation for $L_1$ but in the $y$ direction (which ``returns back''), i.e., given by the operator $M_1=\partial_y + a_1$ (in this case $a_1$ for $L_1$ equals $a$ for $L$). Then the intertwining relation $N_1 L_1 = L_2 M_1$ implies $N_1=\partial_y + a-k_y k^{-1}$, and $L_2=L^{1/k}$.

The composition of these transformations is $(\widehat{M}, \widehat{N})\colon L \rightarrow L^{1/k}$ with $\widehat{M}=M_1M=(\partial_y+a)(\partial_x+b) = \partial_x \partial_y + a\partial_x + b \partial_y +ab + b_y = k+ L$, and $\widehat{N}=N_1 N =k+L^{1/k}=k+L_2$. From here we can readily see how $M_1$ and $N_1$ can be changed to make the composition an identity morphism from~$L$ to~$L$: the inverse Darboux transformation for the Laplace transformation with $M=\dee_x+b$ is $\big(\frac{1}{k}(\dee_y+a_1), \frac{1}{k}(\dee_y-h_{1x} h_1^{-1})\big)\colon L_1 \rightarrow L$.
\end{proof}

(Note that the formulas for the inverse for Laplace transformations can be generalized, as we do below, for Darboux transformations of Type~I, which is a~generalization of Laplace transformations to operators of a more general form.)

{\bf 4.~Intertwining Laplace transformations.} These were introduced in~\cite{ganzha2013intertwining}, and generalize Laplace transformations to linear partial dif\/ferential operators $L \in \K$ of very general form. One starts with any representation $L=X_1X_2-H$, where $L, X_1, X_2, H \in \K$. Then it was proved that there is a Darboux transformation for the operator~$L$,
\begin{gather*}
(X_2,X_2+\omega)\colon \  L \rightarrow L_1 ,
\end{gather*}
where $L_1=X_2X_1+\omega X_1 - H$, and $\omega= -[X_2,H] H^{-1}$. The latter is a pseudodif\/ferential operator in the general case, an element of the skew Ore f\/ield over $\Ko$ that extends~$\K$. In~\cite{ganzha2013intertwining}, E.~Ganzha then adds the requirement that $\omega \in \K$. This is a very general class of transformations and can be used for theoretical investigations. The class contains both invertible and non-invertible Darboux transformations.

{\bf 5.~Darboux transformations of Type~I.} These were introduced by the second author in~\cite{2015:inv:charts}, and are admitted by operators in $\K$ that can be written in the form $L=CM+f$, where $C,M \in \K$, and $f \in \Ko$. We have
\begin{gather*}
\big(M, M^{1/f}\big)\colon \  L \rightarrow L_1 ,
\end{gather*}
where $L_1=M^{1/f}C+f$, writing $M^{1/f}$ for $f M (1/f)$. The inverse Darboux transformation always exists and is as follows:
\begin{gather*}
\left(-\frac{1}{f}  C, -  C \frac{1}{f}\right)\colon \ L_1 \rightarrow L .
\end{gather*}
The original theory of Laplace transformations is very naturally formulated in terms of dif\/fe\-ren\-tial invariants. In~\cite{2015:inv:charts}, analogous ideas were developed for Darboux transformations of Type~I for operators of third order in two independent variables. This can also be done for operators of a~general form using regularized moving frames~\cite{FO1,FO2} and ideas from~\cite{movingframes}. The classical Laplace transformations are special case of transformations of Type~I. We will see that Darboux transformations of Type~I can be identif\/ied with a subclass of the Intertwining Laplace transformations def\/ined above.

Note that the composition of two Darboux transformations of Type~I is (in general) not of Type~I.

Note also that in the Physics literature there are examples of Darboux transformations for concrete dif\/ferential operators such as the non-stationary Schr\"odinger operator or the Fokker--Planck operator, see~\cite{Ioffe_Junker99}.

\section[Operators of general form admitting Wronskian type Darboux transformations]{Operators of general form admitting Wronskian type\\ Darboux transformations}\label{sec:wronski}

Several dif\/ferent types of operators have been proved to admit Wronskian type Darboux transformations. From classical results, it is known that arbitrary 1D operators and 2D Schr\"odinger operator admit Wronskian type Darboux transformations. Also operators in two independent variables of the form $L=\partial_t - a_i \partial_x^i$, for $i=1, \dots, n$ (the Einstein summation convention is used) admit Wronskian type Darboux transformations with $M$ of the form $M=\dee_x+m$~\cite{Matveev79}. Analogues of Wronskian type Darboux transformations are admitted by super Sturm--Liouville operators~\cite{liu95_initial,liu_manas97,liu_manas97_2}, and, as was recently proved, non-degenerate operators on the superline~\cite{2015:super,shemya:voronov2016:berezinians}.

However, there is an abundance of examples where a Wronskian type operator based on a~element of $\ker L$ does not def\/ine a Darboux transformation.
\begin{Example}
For the following operator
\begin{gather*}
L= \partial_t + \partial_x^2 + t \partial_x - \frac{1}{t} ,
\end{gather*}
and the element of its kernel $\psi=t$, $M=\dee_t -\psi_t \psi^{-1}$ does not generate a Darboux transformation. This mean that it is not possible to f\/ind $L_1=\partial_t + a_1 \partial_x^2 + a_2 \partial_x + a_3$ and $N=\partial_t+n$ such that $NL=L_1M$. Thus we ask the following:

\textit{Given an arbitrary operator in $\K$ with $\psi$ in his kernel, when does it admit a~Wronskian type Darboux transformation with $M=\partial_t-\psi_t \psi^{-1}$?}
\end{Example}

Let us take $t$ a f\/ixed but arbitrary variable. We do not assume a priori any special dependence on $\partial_t$. The idea is that we shall start by treating operators $L$ in $\K$ as ordinary dif\/ferential operators with respect to $t$ with coef\/f\/icients in the ring $K[\dee_{x_1}, \dots, \dee_{x_n}]$. Now $K[\dee_{x_1}, \dots, \dee_{x_n}]$ is embedded in $\K$, so we may say that elements of $K[\dee_{x_1}, \dots, \dee_{x_n}]$ are the {\em $t$-free} operators in $\K$.

\begin{Definition}
For a given variable $t$, call a dif\/ferential operator {\em $t$-free} if it does not contain~$\dee_t$ and none of its coef\/f\/icients depend on~$t$.
\end{Definition}

\begin{Theorem}\label{t-free theorem}
Given an operator $L \in K[\dee_t,\dee_{x_1}, \dots, \dee_{x_n}]$, let $M = \dee_t$. Then we have the DT $(M,N) \colon L \rightarrow L_1$ for some operator $L_1$ and a first-order operator $N$ if and only if there exist some $A,B \in K[\dee_t,\dee_{x_1}, \dots, \dee_{x_n}]$ and $c \in K$, $c \neq 0$, such that
\begin{gather*}
L = A\dee_t + cB ,
\end{gather*}
where $B$ is $t$-free. If the Darboux transformation exists, then $L_1$ and $N$ are given by the formulas
\begin{gather*}
L_1 = L - A \dee_t + N A = NA +cB , \\
N = M^{1/c},   \qquad \text{where $M^{1/c}$ is $c M (1/c)$.}
\end{gather*}
Here either $B = 0$, or we may take any non-zero coefficient of our choice in $B$ to be~$1$ by a~suitable choice of~$c$.
\end{Theorem}
\begin{proof}
Let $L$ be given, and assume $N L=L_1 M$ for some $N, L_1 \in \K$, and $M = \dee_t$. Since $\sigma(L) = \sigma(L_1)$, we have $\sigma(M) = \sigma(N)$ so $N=\dee_t + n$, for some $n \in K$.

We shall now construct a certain shift of this Darboux transformation. By division with remainder in $K[\dee_{x_1}, \dots, \dee_{x_n}][\dee_t]$, we can choose $A \in \K$ such that $L'=L - A \partial_t$ does not contain~$\dee_t$. If $L' = 0$, then $L = A \dee_t = AM$ and the conclusion easily follows.
Henceforth assume $L' \neq 0$. We write
\begin{gather*}
L'= \ci \dee_{x_1}^{p_1} \dee_{x_2}^{p_2} \cdots \dee_{x_n}^{p_n} , \qquad p_1, p_2, \dots ,p_n \geq 0 , \qquad \ci \in K ,
\end{gather*}
where we assume summation in upper and lower indices. We also write $L_1'=L_1 - N A \neq 0$. We then get the Darboux transformation
\begin{gather*}
(M,N)\colon \ L' \rightarrow L_1' ,
\end{gather*}
since $NL' = NL -NA\dee_t = L_1 M - NAM = (L_1 - NA)M$. Now, in $N L' = L_1' \partial_t$ every term of $N L'$ must be a left multiple of $\dee_t$, and all the terms that do not contain $\dee_t$ must have their coef\/f\/icients zero. So,
\begin{gather*}
( \dee_t + n )[L'] = 0 ,
\end{gather*}
where the square brackets mean application of $\dee_t + n $ to all the coef\/f\/icients of~$L'$. So, for each term $\ci \dee_{x_1}^{p_1} \dee_{x_2}^{p_2} \cdots\dee_{x_n}^{p_n}$ in $L'$ we must have $(\dee_t + n)[\ci] = 0$. Writing $c$ for $\ci$, we have $c_{t} + n c = 0$ or $-n = c_t/c$. Thus $c_{t}/c = \dee_t(\ln(c))$ is the same for all of the $\ci$, so all the~$\ln(c)$ dif\/fer by functions that do not depend on~$t$, and all the $c$ are equal to within multiplication by functions which do not depend on~$t$. We may also take some non-zero coef\/f\/icient in~$B$ of our choice to be $1$, and write all the~$\ci$ in the form~$c b$ with $b$ $t$-free, using a common function~$c$ that satisf\/ies $c_t/c = -n$, hence
\begin{gather*}
N = \dee_t - c_t/c ,
\end{gather*}
and $L'$ is of the form
\begin{gather*}
L ' = c B ,
\end{gather*}
where $B$ is $t$-free.

We have $N L' = L_1' \partial_t$, and calculate $N L' = (\dee_t - c_t/c)c B = \dee_t c B - c_t B = c \dee_t B + c_t B - c_t B = c \dee_t B = cB \dee_t$, since $B$ is $t$-free. Thus, $L_1'\dee_t = cB \dee_t$, and cancelling $\dee_t$ gives us $L_1' = cB$. So
\begin{gather*}
L_1' = cB = L' .
\end{gather*}
Thus, $L = L' + A \partial_t=A \partial_t+cB$. Also $L_1=L_1' + N A= cB+NA$.

The other direction is an easy calculation.
\end{proof}

This theorem also gives us a similar result for general transformations of Wronskian type.

\begin{Theorem}[criterion]\label{thm:wronski_type_char}
Given an operator $L \in \K$ and a $\psi \in \ker L$, the operator $M = \partial_t-\psi_t \psi^{-1}$ generates a DT if and only if there exists some $A \in \K$, and a $t$-free $B \in \K$ such that
\begin{gather*}
L^{\psi} = A \partial_t + c B .
\end{gather*}
Here either $B = 0$, or any one non-zero coefficient in $B$ may be taken to be $1$. If the DT exists, then $L_1$ and $N$ are given by $L_1^{\psi} = L^{\psi} - A \partial_t +N^{\psi}A$, $N^\psi = \partial_t^{1/{c}}$, where $\dee_t^{1/c}$ is $c \dee_t (1/c)$.
\end{Theorem}
\begin{proof}
Let $L$ be given, together with $\psi \in \ker(L)$, and let $M$ be $\dee_t - \psi_t\psi^{-1}$. Note that $M^\psi = \psi^{-1} (\dee_t -\psi_t \psi^{-1}) \psi =
\psi^{-1} (\psi \dee_t + \psi_t - \psi_t) = \dee_t$, so we conjugate $N L = L_1 M$ by $\psi$. This gives $N^\psi L^\psi = L_1^\psi M^\psi = L_1^\psi \dee_t$. Now apply the previous theorem.
\end{proof}

\begin{Corollary}[necessary condition]\label{cor:1}
Let an operator $L \in \K$, and $\psi \in \ker L$ be given, where $M_\psi = \partial_t-\psi_t \psi^{-1}$ generates a DT. Then there exists some $A \in \K$, and a $t$-free $B \in \K$ such that
\begin{gather*}
L = A M_\psi + c B , \qquad [M_\psi,B]=0 .
\end{gather*}
Here $B$ is not necessarily $t$-free.
\end{Corollary}
\begin{proof} From the previous theorem, $L^\psi = A \dee_t + cB$ and $[\dee_t,B]$. Conjugating by $\psi^{-1}$, $L = A^{\psi^{-1}} \dee_t^{\psi^{-1}} + c^{\psi^{-1}} B^{\psi^{-1}}$, with $[\dee_t^{\psi^{-1}},B^{\psi^{-1}} ] = 0$. Note that $\dee_t^{\psi^{-1}} = \psi \dee_t \psi^{-1} =\psi(\psi^{-1} \dee_t - \psi_t \psi^{-2}) = \dee_t - \psi_t \psi^{-1} = M_\psi$, so $L = A^{\psi^{-1}} M_\psi + c B^{\psi^{-1}}$, and $[M_\psi,B^{\psi^{-1}}] = 0$.
Now rename $A^{\psi^{-1}}$ as~$A$, and~$B^{\psi^{-1}}$ as~$B$.
\end{proof}

\section{Classif\/ication of Darboux transformations of f\/irst order}\label{sec:1storder_classification}

We know that Type~I Darboux transformations are never of Wronskian type, because they have $\ker L \cap \ker M = \{0\}$. For f\/irst-order transformations, we can say more.

\begin{Theorem}\label{thm:classification_order_one}
Suppose there is a first-order Darboux transformation $(M,N)\colon L \rightarrow L_1$ with $N,M,L_1 \in \K$, and $M=\partial_t+m$. Then this transformation is either of Type~I or of Wronskian type.
\end{Theorem}
\begin{proof}
While $M = \dee_t + m$ for some $m \in \Ko$, it is more convenient to write~$M$ as $\dee_t^{v^{-1}} = v \dee_t v^{-1} = \dee_t -v_tv^{-1}$, which is easily done by solving $-mu = v_t$.

Then we have $N L = L_1 \dee_t^{v^{-1}}$, and conjugate by $v$ to get $N^v L^v = L_1^v \dee_t$. By Theorem~\ref{t-free theorem}, we have $L^v = A \dee_t + cB$, where $B$ is $t$-free. We have three cases.

If $cB=0$ then $L = A^{v^{-1}} M$ and we have a Darboux transformation obtained from a factorization. Then $v \in \ker L \cap \ker M$, and $M = \dee_t - v_t v^{-1}$, showing the DT is of Wronskian type. Henceforth, assume $cB \neq 0$.

Suppose f\/irst that $cB$ is a non-zero function. Then we may take $cB = c$, since we are free to pick one coef\/f\/icient of~$B$ to be~$1$. Then $L^v = A \dee_t + c$, and conjugation gives $L = v A v^{-1} M + c$. We let $C$ be $v A v^{-1}$ and $f$ be $c$, giving $L = CM + f$. Theorem~\ref{t-free theorem} also gives $N^v = \dee_t - c_t c^{-1} = c \dee_t c^{-1}$, so $N = c v \dee_t v^{-1} c^{-1} = f M f^{-1}$, as expected. Similarly, $L_1^v = L^v - A \partial_t + ( \partial_t -c_{t} c^{-1} ) A = L^v - A \dee_t + N^v A$, yielding $L_1 = L - v A v^{-1} M + N v A v^{-1} = CM + f - CM + NC = NC + f$. This shows the transformation is of Type~I.

We now suppose that $B$ is a $t$-free dif\/ferential operator that is not a function. Then we have a non-zero $t$-free function $\phi$ in $\ker(cB)$. We let $\psi$ be $v \phi$, and claim $\psi$ is in $ \ker L$. We have $L = v A \dee_t v^{-1} + v cB v^{-1}$, so $L[\psi]$ is $(v A \dee_t v^{-1} + v cB v^{-1})[v \phi] = v A \dee_t[\phi] + vcB[\phi] = 0$, because $\phi$ is $t$-free. The Wronskian operator for $\psi$ is $\dee_t - \psi_t \psi^{-1} = \dee_t - (v \phi)_t (v \phi)^{-1} = \dee_t - v_t \phi v^{-1} \phi^{-1} = \dee_t - v_t v^{-1} = M$. This shows the transformation is of Wronskian type.
\end{proof}

We can put f\/irst-order Darboux transformations of all types on a common footing, and show that they depend on the existence of a representation of $L$ in a special form.

\begin{Theorem} \label{thm:commutator}\quad
\begin{enumerate}\itemsep=0pt
\item[$1.$] 
A Darboux transformation $(M,N)\colon L \rightarrow L_1$ where $M$, $N$ are first order exists if and only if $L = CM + cB$ for some operators $C$ and $B$ and some $c \in \Ko$, where $[M,B] = 0$.
\item[$2.$] 
If $L$, $M$, $C$, $B$ and $c$ are as above, we have the following:
\begin{itemize}\itemsep=0pt
\item If $cB=0$ then $(M,N)$ is a Darboux transformation obtained from a factorization, and is also of Wronskian type.
\item If $cB \in K$ and $cB \neq 0$ then $(M,N)$ is a transformation of Type~I.
\item If $cB \notin K$ $($is an operator of order larger than zero$)$ then $(M,N)$ is a transformation of Wronskian type.
\end{itemize}
The corresponding operators $N$ and $L_1$ for the Darboux transformation are
\begin{gather} \label{eq:formulas:forNL1:1st:order}
N = c M c^{-1} , \qquad L_1=NC+cB , \qquad \text{if} \quad cB \neq 0,
\end{gather}
 and in the case $cB=0$, $N$ is any operator with $\sigma(N) = \sigma(M)$, and $L_1=NC$.
\end{enumerate}
\end{Theorem}
\begin{proof}
For (1), we f\/irst assume that $L=CM + cB$, where $MB = BM$. If $cB = 0$, then $L = CM$, and taking any $N$ with $\sigma(N) = \sigma(M)$, and $L_1 = NC$, we have $NL = NCM = L_1M$. Now assuming $cB \neq 0$, we let $N = cMc^{-1}$, and $L_1 = NC + cB$. This gives us $NL = cMc^{-1}(CM + cB)
= cMc^{-1}CM + cMB = NCM +cBM = (NC+cB)M = L_1 M$. Since $\sigma(M) = \sigma(N)$, we get $\sigma(L) = \sigma(L_1)$ showing that this is a Darboux transformation.

Now let us prove the statement in the opposite direction and assume $(M,N) \colon L \rightarrow L_1$ is a DT with $M$ and $N$ f\/irst order. As before we may change variables and apply gauge transformations to make $g^{-1} M g =\dee_u$. Conjugating $NL = L_1 M$, we get
$N^g L^g = L_1^g \dee_u$. By Theorem~\ref{t-free theorem}, this gives $L^g = C' \dee_u + c B'$, where $B'$ is $u$-free. Thus $[\dee_u,B'] = 0$, which implies $[M,gB'g^{-1}] = 0$. Then $L^g = C' \dee_u + c B'$ becomes $L = g C' g^{-1} M + c g B' g^{-1}$, which we write as $L = C M + c B$ by
taking $C = g C' g^{-1}$ and $B = g B' g^{-1}$.

For (2), note that when we changed variables and applied a gauge transformation in the proof of (1) to make $g^{-1} M g =\dee_u$, that the operator $\dee_u$ we obtained was unique. Then the $C'$ and $cB'$ we obtained were unique, and thus the $C$ and $cB$ in $L = CM + cB$ are uniquely determined.

In our f\/irst case, we assume $cB \notin K$. As before, we change variables and apply a gauge transformation to get $L^g = C'\dee_u + cB'$, where $cB' \notin K$. As in the proof of Theorem~\ref{thm:classification_order_one}, we have a non-zero $u$-free function $\phi$ in $\ker(cB')$, since $B'$ is $u$-free. We now let $\psi$ be $g \phi$, and have that~$\psi$ is in~$ \ker L$. The Wronskian operator for $\psi$ and $\dee_u$ is then $\dee_u - \psi_u \psi^{-1} = \dee_u - (g_u \phi)(g \phi)^{-1} = \dee_u - g_u g^{-1} = g \dee_u g^{-1} = M$, so $M$ is obtained from a Wronskian operator by a change of variables, and the transformation is of Wronskian type by Def\/inition~\ref{Wronskian definition}.

Next assume $cB = 0$, so $L = CM$. Taking any $N$ with $\sigma(N) = \sigma(M)$, we let $L_1$ be $NC$ and have $NL = NCM = L_1M$, making the DT one obtained from a factorization. The same DT is also of Wronskian type. We change variables and transform, getting $L^g = C'\dee_u$, where $M = g \dee_u g^{-1}$.
Then $1$ is in the kernels of $L^g$ and $\dee_u$, and $g$ is in the kernels of $L$ and $M$. As in the previous paragraph, $\dee_u - g_u g^{-1} = g \dee_u g^{-1} = M$, showing the transformation is of Wronskian type.

Finally, assume $cB = f$ is a non-zero function in $K$. Then $L = CM + f$. If $N = f M f^{-1}$, then $NL = NCM + Nf = NCM + fM = (NC + f)M$ and the DT is of Type~I.

We must also show that when $cB \neq 0$, that $N$ must be $cMc^{-1}$. To see this, we have $L^g = C'\dee_u + cB'$, as before. Considering $(\dee_u,N') \colon L^g \rightarrow L_1'$, we get from Theorem~\ref{t-free theorem} that $N'$ must be $c \dee_u c^{-1}$. Transforming back to the DT $(M,N) \colon L \rightarrow L_1$, $N = g N' g^{-1} = g c \dee_u c^{-1} g^{-1} = c g \dee_u g^{-1} c^{-1} = c M c^{-1}$, as desired. \end{proof}

\begin{Remark}
As Type~I Darboux transformations do, the transformations of Wronskian type of order one f\/it the framework developed by E.I.~Ganzha in~\cite{ganzha2013intertwining}. Given $L = CM + cB$ with $[M,B] = 0$, we initially assume $cB \neq 0$, and write $L = X_1 X_2 - H$ by taking $X_1 = C$, $X_2 = M$ and $H = -cB$. We need $\omega = -[X_2,H] H^{-1}$ to be a strictly dif\/ferential operator, and calculate
$\omega = -[M,-cB] (-cB)^{-1} = (c B M - M c B) B^{-1} c^{-1} = (c M B - M c B) B^{-1} c^{-1} = (c M - M c) c^{-1} = c M c^{-1} - M = N - M$. This also gives us that $N$ is $M + (N - M) = X_2 + \omega$, as required.

In case $cB = 0$, then $L = CM = (C-1)M + M$, and we take $X_1 = C-1$, $X_2 = H = M$, which works.
\end{Remark}

\begin{Remark}
Let us call an expression of $L$ as $L = CM + cB$, with $[M,B]=0$ and $c \in \Ko$ a~{\em quasi-factorization} of~$L$. A natural question is which operators $L$ have quasi-factorizations, and if they do, how many do they have? We do not have uniqueness of factorization for partial dif\/ferential operators, as, e.g., shown by Landau's example (mentioned in~\cite{Blumberg} as well as in~\cite{tsarev:2009:inbook} and our Example~\ref{from_Landau}). Thus the general question of how many quasi-factorizations an operator~$L$ has may be dif\/f\/icult. We may however note the following for Darboux transformations of Type~I, where $cB = c$. For the quasi-factorization $L = CM + c$ where~$C$ and~$M$ are not functions, we have that the principal symbols of $C$ and $M$ are factors of the principal symbol of $L$. If $M$ is f\/irst order, it is almost determined by its principal symbol. Given a f\/irst-order~$M$, we may change variables if need be, and write~$M$ as $a \dee_t + b$. Grouping $a$ with~$C$, we may assume that the principal symbol of $M$ is $p_t = \sigma(\dee_t)$. A quasi-factorization of~$L$ is determined by~$M$, so it is natural to ask if there can be more than one f\/irst-order quasi-factor~$M$ of~$L$ with the same principal symbol~$p_t$.

For Darboux transformations of Type~I, the answer is no, unless the principal symbol of $L$ has a repeated factor of $p_t$. For suppose we have two f\/irst-order quasi-factors $M_1$ and $M_2$ of $L$, that both have principal symbol~$p_t$. Applying an appropriate gauge transformation, we may assume $M_1 = \dee_t$ and $M_2 = \dee_t + n$, where $n \neq 0$. Then $L = C \dee_t + f = E(\dee_t + n) + g$, where $f$ and $g$ are in $\Ko$. Rearranging, $En = (C-E)\dee_t + f - g$. Thus $p_t$ divides the principal symbol of $En$, and hence of~$E$. Thus $p_t^2$ divides the principal symbol of~$L$. This gives us that when the principal symbol of $L$ has no repeated factors, that there is at most one f\/irst-order Type~I transformation per factor, and thus at most $\deg(L)$ many f\/irst-order Type~I Darboux transformations of~$L$.
\end{Remark}

In the case of Darboux transformations which are not of Type~I, the situation is not as simple, since the principal symbol of $M$ does not necessarily divide that of~$L$. See the following example.

\begin{Example}\label{x dee_x}
We let $L$ be
\begin{gather*}
\dee_{xxy} + \dee_{xyy} + (1-x/2)\dee_{xx} + (3-x)/2 \dee_{xy} \\
\qquad{} + \big({-}1/x+1/2x^2\big)\dee_{yy} + 1/2\dee_x + \big({-}1/x+1/x^2\big)\dee_y.
\end{gather*}
Letting
\begin{gather*}
M = x \dee_x + \dee_y
\end{gather*}
we have $L = CM + cB$, where $BM = MB$. Here $c = x(x-1)/(8e^{3y})$, $C = (1/8)((1-x)\dee_{xx} +(4 +4/x)\dee_{xy}+(1/x-1/x^2)\dee_{yy} + (1+3/x)\dee_x
-(2/x+2/x^2)\dee_y)$, and $B = e^{3y}x^{-3}(x^3\dee_{xxx} - 3x^2\dee_{xxy} + 3x\dee_{xyy} - \dee_{yyy} - 3x^2\dee_{xx} + 9x\dee_{xy} - 6\dee_{yy} + 3x\dee_x -8\dee_y)$.
\end{Example}

Another natural question about Darboux transformations is to what extent $L$ and $M$ determine $L_1$ and $N$. As a partial answer, we have the following.

\begin{Definition}\label{uniquely determine definition}
We say that operators $L$ and $M$ {\em uniquely determine} their Darboux transformation if\/f there is at most one pair of operators $N$ and $L_1$ so that $(M,N)$ is a~DT from~$L$ to~$L_1$.
\end{Definition}

\begin{Theorem}\label{L and M determine theorem}
If $M \notin \Ko$ is a first-order operator, and $L$ is any operator, then $L$ and $M$ uniquely determine their Darboux transformation iff~$L$ can not be written as $L = AM$ for any operator~$A$.
\end{Theorem}
\begin{proof}
Let $M \notin \Ko$ be a f\/irst-order operator. If $L = AM$, we may take any $N$ with $\sigma(N) = \sigma(M)$, let $L_1 = NA$, and have a Darboux transformation. For the other direction, assume that $L \neq AM$ for every operator $A$, and that we have $N_1$, $N_2$, $L_1$ and $L_2$ that give two distinct DTs, so we have $N_1 L = L_1 M$, $N_2 L = L_2 M$, and $\sigma(N_1) = \sigma(M) = \sigma(N_2)$. If $N_1 = N_2$, then $L_1M = N_1 L = N_2 L = L_2 M$, and $L_1 = L_2$ by cancellation. So we assume that $N_1 \neq N_2$. We have $\sigma(N_1) = \sigma(N_2)$ where $N_1$ and $N_2$ are f\/irst order, so $N_1 - N_2$ is a~nonzero function $f$. Now we subtract intertwining relations, and get $(N_1 - N_2) L = N_1L -N_2L = L_1M - L_2M = (L_1 - L_2)M$. Then letting $A = f^{-1}(L_1 - L_2)$, we have $AM = f^{-1}(L_1 - L_2)M = f^{-1}(N_1 - N_2)L = L$, a~contradiction.
\end{proof}

The following example is derived from Landau's famous example~\cite{Blumberg,tsarev:2009:inbook} of non-unique factorization of linear partial dif\/ferential operators. The resulting Darboux transformation of higher order exhibits a dif\/ferent pattern than Darboux transformations of order one.

\begin{Example} \label{from_Landau}
Landau's example of non-unique factorization of operators is
\begin{gather*}
R Q =Q Q P,
\end{gather*}
where $P = \dee_x + x \dee_y$, $Q = \dee_x + 1$ and the operator $R = \dee_{xx} + x \dee_{xy} + \dee_x + (2+x)\dee_y$ is irreducible over any extension of the dif\/ferential f\/ield of rational functions in $x$ and $y$.

This gives us a Darboux transformation,
\begin{gather*}
(M,N)\colon \ Q \rightarrow Q, \qquad \text{where} \quad M = QP \ \  \text{and} \  \ N = R .
\end{gather*}
Now formula~\eqref{eq:formulas:forNL1:1st:order} in not satisf\/ied for this $M$ and $N$. Indeed, assuming $N = c M c^{-1}$ or $R = cQPc^{-1}$, we obtain
$R = \dee_{xx} + x \dee_{xy} + (1-2c_xc^{-1}-xc_yc^{-1}) \dee_x + (x-xc_xc^{-1} +1) \dee_y$, which implies $2c_xc^{-1}+xc_yc^{-1}=0$ and $xc_xc^{-1} = -1$. The latter gives $c = f/x$ where $f$ does not depend on $x$, but plugging this into the former gives $c = 0$.
\end{Example}

\section{Continued Type~I Darboux transformations}\label{sec:it}

The following is an example of \textit{a second-order Darboux transformation that is not of Type~I, nor obtained from a factorization, nor a multiple of a DT of Wronskian type}.

\begin{Example} \label{new type} Let $M = \dee_{xx} + 1$, $F = \dee_x$, $C = \dee_y + x$, $L = CM + F$ and $L_1 = MC + F$, and consider the Darboux transformation
\begin{gather*}
(M,M) \colon \  L \rightarrow L_1 .
\end{gather*}
Indeed, $NL = M(CM + F) = MCM + MF =MCM + FM = (MC + F)M = L_1M$.

Observe that for any operators $A$, $B$ and $C$, that $\ker A \cap \ker (B + CA) = \ker A \cap \ker B$. Thus $\ker L \cap \ker M = \ker (CM + F) \cap \ker M =
\ker F \cap \ker M = \ker F \cap \ker (M - \dee_x F) = \ker F \cap \ker 1 = \{0\}$, so this transformation is not obtained from a factorization, nor a multiple of a DT of Wronskian type. $F$ is not in $\Ko$, so the transformation is also not of Type~I.
\end{Example}

Our long term goal is to classify Darboux transformations for an operator of general form. The example above is of neither type known to us. It can be understood in terms of the following general phenomenon. Suppose that we have operators $L$ and $M$ so that $L = CM + F$ for some operators $C,F \notin \Ko$. So the corresponding Darboux transformation cannot be of Type~I. We can rewrite this as $F = L - CM$, where we have written $F$ as a particular kind of linear combination of~$L$ and~$M$. If $F$ were a function~$f$, we could get a~Type~I transformation where $NL = L_1M$. Now suppose $F$ is not a function, but that we have $M = AF + f$ for some operator $A$ and nonzero function~$f$. We can write this as $f = M - AF$, and view this as reaching $f$ in two linear
combination steps. The f\/irst gives~$F$, and the second gives~$f$.

As in Example \ref{new type}, we have $\ker L \cap \ker M = \ker (CM + F) \cap \ker M = \ker F \cap \ker M = \ker F \cap \ker (AF - f) = \ker F \cap \ker f = \{0\}$. Thus a Darboux transformation with this $L$ and $M$ cannot be obtained from a factorization, nor is it a multiple of a DT of Wronskian type.

These new Darboux transformations constitute a class which is appreciably ``larger'' than that of Type~I transformations.

\begin{Example} \label{dee_xxy example}
Consider operators on the two variables $x$ and $y$, and let $L$ have the form $a_{001}\dee_{xxy} + a_{00}\dee_{xx} + a_{01}\dee_{xy} + a_0\dee_x + a_1\dee_y + a$. Not every such $L$ has a Type~I transformation $(M,N) \colon L \rightarrow L_1$, for if we for instance write $L = CM + f = (b_{00}\dee_{xx} + b_0\dee_x + b)(c_1\dee_y + c) + f$ we get $a_{001} = b_{00}c_1$ and f\/ive other equations for the f\/ive unknown functions $b_{00}$, $b_0$,~$b$,~$c_1$ and~$c$. Such a~system usually has no solution, and trying the other possible forms of $C$ and $M$ is little help.

On the other hand, almost every such $L$ does have a DT $(M,N) \colon L \rightarrow L_1$. We will write $L = CM + F$, where $M = AF + f$ and $A$ is a f\/irst-order operator. Then letting $N = f F f^{-1} A + f$, we have $\sigma(N) = \sigma(M)$ and $NL = N(CM + F) = NCM + (f F f^{-1} A + f) F =
NCM + f F f^{-1}A F + f F = NCM + f F f^{-1}A F + f F f^{-1} f = NCM + (f F f^{-1})(AF + f) = (NC + f F f^{-1})M$. Setting $L_1 = NC + f F f^{-1}$, this is the desired intertwining relation.

We now let $C = g\dee_y + h$, $F = p\dee_x + q$, and $A = b\dee_x + c$, where $g$, $h$, $p$, $q$, $b$ and $c$ are unknown functions. We want $L = CM + F = C(AF + f) + F = (g\dee_y + h)((b\dee_x + c)(p\dee_x + q) + f) + p\dee_x + q = (g\dee_y + h)((bp\dee_{xx} + r\dee_x + s) + p\dee_x + q$, where we write $r = cp + bq + bp_x$ and $s = cq + bq_x + f$. Multiplying this out and equating coef\/f\/icients, we obtain $a_{001} = gbp$, $a_{00} = hbp + g(bp)_y$, $a_{01} = rg$, $a_0 = hr + gr_y + p$, $a_1 = gs$, and $a = hs + gs_y + q$. Nothing is lost by setting $bp = 1$, so we do so. This gives $g = a_{001}$ and $h = a_{00}$.
Then $r = a_{01}/g = a_{01}/a_{001}$ and $s = a_1/g = a_1/a_{001}$. Thus $p$ and $q$ are determined, as $p = a_0 - hr - gr_y$ and $q = a - hs - gs_y$, and $bp = 1$ gives $b = 1/p$. Now we get $c$ and $f$ from $r$ and $s$, giving $c = (r -bq -bp_x)/p$ and then $f = s - cq -bq_x$.
\end{Example}

While the above examples had two linear combination steps, this can be extended to any number of linear combination steps, giving the following theorem.

\begin{Theorem}\label{Types_above_1_theorem}
Let $L \in \K$, and suppose there are nonzero operators $A_1, A_2,$ $\dots, A_k$, $M = M_1, M_2, \dots, M_k \in \K$ for some $k \geq 1$ and a nonzero \mbox{$f = M_{k+1} \!\in\! \Ko$} so that $M_{i-1} = A_i M_i + M_{i+1}$, $1 \leq i \leq k$, where $M_0=L$, i.e.,
\begin{gather*}
L = A_1 M_1 + M_{2}, \\
M_1  = A_2 M_2 + M_3, \\
\cdots\cdots\cdots\cdots\cdots\cdots \\
M_{i-1} = A_i M_i + M_{i+1}, \qquad \text{for} \quad 1 \leq i \leq k, \quad \text{and finally} \\
\cdots\cdots\cdots\cdots\cdots\cdots \\
M_{k-1}  = A_k M_k + f.
\end{gather*}

Then there exists a Darboux transformation for operator $L$
\begin{gather*}
(M_1,N)\colon \  L \rightarrow L_1
\end{gather*}
defined as follows. Define $N_{k+1} = M_{k+1} = f$ and$N_{k} = fM_kf^{-1}$, and define $N_i$ for $0 \leq i \leq k-1$ by downward recursion using
\begin{gather*}
N_i = N_{i+1}A_{i+1} + N_{i+2} .
\end{gather*}
Finally let $N = N_1$, and let $L_1 = N_0$. Then $M$ and $N$ have the same principal symbol, and the intertwining relation $NL = L_1M$ holds. The corresponding Darboux transformation is not obtained from a factorization, nor a multiple of a DT of Wronskian type, and if $k > 1$ it is not of Type~I.
\end{Theorem}
\begin{proof}
Fix $k$. If $k = 1$, we have $M_2 = f$, so $L = M_0 = A_1M_1 + M_2 = A_1M_1 + f$. Then $N = N_1 = fM_1f^{-1}$, and $L_1 = N_0 = N_1A_1 + N_2 = fM_1f^{-1}A_1 + f$. We have $NL = L_1M$, since this is a~Darboux transformation of Type~I.

Henceforth assume $k > 1$. We have that $\ker(L) \cap \ker(M) = \ker(M_0) \cap \ker(M_1) \subseteq \ker(M_1) \cap \ker(M_2)$, the last step because $M_2$ is
a linear combination of $M_0$ and $M_1$. Now $M_{i-1} = A_i M_i + M_{i+1}$ for all $i$, so each $M_{i+1}$ is a~linear combination of $M_{i-1}$ and
$M_i$, giving us similarly that $\ker(M_{i-1}) \cap \ker(M_i) \subseteq \ker(M_i) \cap \ker(M_{i+1})$ for all $1 \leq i \leq k$. Thus $\ker(M_0) \cap \ker(M_1) \subseteq \ker(M_1) \cap \ker(M_2) \subseteq \ker(M_2) \cap \ker(M_3) \subseteq \cdots\subseteq \ker(M_k) \cap \ker(M_{k+1})$. But $\ker(M_{k+1}) = \ker(f) = \{ 0 \}$, so $\ker(L) \cap \ker(M) = \{ 0 \}$, and our transformation will not be obtained from a~factorization, nor a~multiple of a DT of Wronskian type.

We will now do two inductive proofs of statements $S_i$, starting with $i=k$ as the basis and showing $S_i \Rightarrow S_{i-1}$ for all $i \geq 1$.

We f\/irst show $N_i M_{i-1} = N_{i-1}M_i$ for $1 \leq i \leq k+1$. When $i=1$, we have $N = N_1$, $L = M_0$, $L_1 = N_0$ and $M = M_1$, so this will be $NL = L_1M$. Our basis is when $i=k+1$. Then we have $N_i M_{i-1} = N_{k+1} M_k = f M_k = f M_k f^{-1} f = N_k M_{k+1} = N_{i-1}M_i$, as desired. Now let $1 < i \leq k+1$, and assume $N_i M_{i-1} = N_{i-1}M_i$. Adding $N_{i-1}A_{i-1}M_{i-1}$ to both sides, we get $N_i M_{i-1} + N_{i-1}A_{i-1}M_{i-1} = N_{i-1}M_i + N_{i-1}A_{i-1}M_{i-1}$. But $N_i + N_{i-1}A_{i-1} = N_{i-2}$ and $M_i + A_{i-1}M_{i-1} = M_{i-2}$, so we have $N_{i-2}M_{i-1} = N_i M_{i-1} + N_{i-1}A_{i-1}M_{i-1} = N_{i-1}M_i + N_{i-1}A_{i-1}M_{i-1} = N_{i-1}M_{i-2}$. Reversing the order of equality, $N_{i-1}M_{i-2} = N_{i-2}M_{i-1}$, which is the required statement with $i$ replaced by $i-1$ throughout. Thus $N_i M_{i-1} = N_{i-1}M_i$ for $1 \leq i \leq k+1$ by induction.

We can now show that $M_i$ and $N_i$ have the same principal symbols for $ 0 \leq i \leq k+1$, which will give $\sigma(L) = \sigma(L_1)$ since $M_0 = L$ and $N_0 = L_1$. To start with, we have $N_{k+1} = f = M_{k+1}$, so $\sigma(N_{k+1}) = \sigma(M_{k+1})$. Now we let $0 \leq i \leq k+1$, and assume $\sigma(N_i) = \sigma(M_i)$. Since $N_i M_{i-1} = N_{i-1}M_i$ from the previous paragraph, we get $\sigma(M_{i-1}) = \sigma(N_{i-1})$. Thus the desired conclusion follows by induction.
\end{proof}

This theorem then gives us a new type of Darboux transformation, which generalize Type~I transformations.

\begin{Definition}\label{iterated type definition}
Darboux transformations def\/ined as in Theorem~\ref{Types_above_1_theorem} we shall call {\em continued Type~I Darboux transformations}.
\end{Definition}

A natural question is how Darboux transformations of continued Type~I relate to the theory developed by E.I.~Ganzha in~\cite{ganzha2013intertwining}. Let us consider transformations of continued Type~I for $k = 2$, so we have $L = A_1 M_1 + M_2$ and $M_1 = A_2 M_2 + M_3$ where $M_3 = f$. Intertwining Laplace transformations depend on writing $L$ in a particular form, where $L = X_1 X_2 - H$. In our notation, the def\/inition from~\cite{ganzha2013intertwining} then gives $M = X_2$, $\omega = -[ X_2, H ] H^{-1}$ and $N = X_2 + \omega$, where we would like $\omega$ to be a dif\/ferential operator and not merely pseudo-dif\/ferential. Substituting $M = M_1$ for $X_2$, we have $L = A_1 M_1 + M_2 = X_1 X_2 - H$. So it would be natural to try setting $X_1 = A_1$, giving $H = -M_2$. In this case, we get $\omega = -[ M_1, H ] H^{-1} = -(M_1 (-M_2) - (-M_2) M_1)(-M_2)^{-1} = (M_2 M_1 - M_1 M_2)M_2^{-1}$. Expanding this using $M_1 = A_2 M_2 + f$, it becomes $(M_2(A_2 M_2 +f)-(A_2 M_2 +f)M_2)M_2^{-1} = ((M_2 A_2 - A_2 M_2 - f)M_2 + M_2 f)M_2^{-1}$, which is not a~dif\/ferential operator unless $M_2$ and $f$ commute, which is seldom true.

We have that transformations of continued Type~I are not obtained from a factorization, nor a multiple of a DT of Wronskian type. We also note that when $M$ is f\/irst order, that the continued Type~I Darboux transformation $NL = L_1M$ must be of Type~I. It turns out that all continued Type~I DTs are invertible.

Some insight into why this is can be obtained by looking at the construction of the continued Type~I DT $(M_1,N_1) \colon L \rightarrow L_1$ by alternately forming shifts and duals of a simple starting DT. With notation as in Theorem~\ref{Types_above_1_theorem}, we have $M_{k+1} = N_{k+1} = f$, and $N_k = f M_k f^{-1}$. Then $(M_k,N_k) \colon M_{k+1} \rightarrow N_{k+1}$ is an invertible DT, since it is Type~I. Now $(M_k,N_k) \colon M_{k+1} + A_k M_k \rightarrow N_{k+1} + N_k A_k$ is a shift of $(M_k,N_k) \colon M_{k+1} \rightarrow N_{k+1}$, and also an invertible DT. But $M_{k+1} + A_k M_k = M_{k-1}$ and
$N_{k+1} + N_k A_k = N_{k-1}$, so this DT is actually $(M_k,N_k) \colon M_{k-1} \rightarrow N_{k-1}$. Then its dual $(M_{k-1},N_{k-1}) \colon M_k \rightarrow N_k$ is also an invertible DT. Now we shift again, getting that $(M_{k-1},N_{k-1}) \colon A_{k-1} M_{k-1} + M_k \rightarrow N_{k-1} A_{k-1} + N_k$ or
$(M_{k-1},N_{k-1}) \colon M_{k-2} \rightarrow N_{k-2}$ is an invertible DT. Then we take the dual of this, and continue the process, eventually obtaining that
$(M_1,N_1) \colon M_0 \rightarrow N_0$ or $(M,N) \colon L \rightarrow L_1$ is an invertible DT.

The above process is constructive, in the sense that the inverse of $(M,N) \colon L \rightarrow L_1$ can be recovered using the descriptions of the inverses of a shift and a dual in Lemmas~\ref{shift invertibility lemma} and~\ref{dual is invertible lemma}. Doing this gives recursive def\/initions, as in the inductive proof of the following result.

\begin{Theorem}\label{iterated invertible theorem}
All continued Type~I Darboux transformations are invertible.
\end{Theorem}
\begin{proof}
Let $NL = L_1M$ be of continued Type~I, and let $k$, $L = M_0$, $A_1, A_2, \dots, A_k$, $M = M_1, M_2, \dots, M_k$ and a nonzero function $f = M_{k+1}$ be as in Theorem~\ref{Types_above_1_theorem}. As there, we have $M_{i-1} = A_i M_i + M_{i+1}$ for $1 \leq i \leq k$, and also def\/ine $N_i$ by $N_{k+1} = M_{k+1} = f$, $N_{k} = fM_kf^{-1}$, and $N_i = N_{i+1}A_{i+1} + N_{i+2}$. for $0 \leq i \leq k-1$.

Using $M_{i-1} = A_i M_i + M_{i+1}$ repeatedly, we can write $L = M_0 = P_kM_k + Q_kM_{k+1}$. The operator $P_k$ is of interest to us, so we def\/ine it inductively as follows. We let $P_0 = 1$, $P_1 = A_1$, and $P_{i+1} = P_iA_{i+1} + P_{i-1}$ for all $i \geq 1$. For example, $P_2 = P_1A_2 + P_0 = A_1A_2 + 1$.

We also need related operators $P'_i$, created by reversing the order of all the products in $P_k$. For example, $P'_2$ is $A_2A_1 + 1$. We def\/ine these
inductively by letting $P'_0 = 1$, $P'_1 = A_1$, and $P'_{i+1} = A_{i+1}P'_i + P'_{i-1}$ for all $i \geq 1$.

We need the fact that $P_iP'_{i+1} = P_{i+1}P'_i$ for all $i \geq 0$, and prove this by induction. For the basis, we have $i = 0$ and get $P_0P'_1 = 1 A_1 = A_1 1 = P_1P'_0$. Now assume that $i \geq 1$ and that $P_{i-1}P'_i = P_iP'_{i-1}$. We get $P_iP'_{i+1} = P_i(A_{i+1}P'_i + P'_{i-1}) = P_iA_{i+1}P'_i + P_iP'_{i-1} = P_iA_{i+1}P'_i + P_{i-1}P'_i = (P_iA_{i+1} + P_{i-1})P'_i = P_{i+1}P'_i$, which proves the claim.

We have $L = M_0 = P_kM_k + P_{k-1}M_{k+1}$, and show this by proving that $M_0 = P_iM_i + P_{i-1}M_{i+1}$ for $1 \leq i \leq k$ by induction on~$i$. Our basis is where $i = 1$. Here we have $M_0 = A_1M_1 + 1 M_2 = P_1M_1 + P_0M_2$. Now assume that $M_0 = P_iM_i + P_{i-1}M_{i+1}$, we will show $M_0 = P_{i+1}M_{i+1} + P_iM_{i+2}$. We have $M_0 = P_iM_i + P_{i-1}M_{i+1} = P_i(A_{i+1}M_{i+1} + M_{i+2}) + P_{i-1}M_{i+1} = (P_iA_{i+1}+P_{i-1})M_{i+1} + P_iM_{i+2}= P_{i+1}M_{i+1} + P_iM_{i+2}$, as required.

A similar inductive proof shows that $L_1 = N_0 = N_k P'_k + N_{k+1}P'_{k-1}$. We leave it to the reader.

To construct the inverse transformation $N'L_1 = LM'$, we let $N' = (-1)^k P_k f^{-1}$ and $M' = (-1)^k f^{-1}P'_k$. Then $N'L_1 =
(-1)^k P_k f^{-1}(N_k P'_k + N_{k+1}P'_{k-1}) = (-1)^k P_k f^{-1}(f M_k f^{-1} P'_k + f P'_{k-1}) = (-1)^k P_k ( M_k f^{-1} P'_k + P'_{k-1}) = (-1)^k (P_k M_k f^{-1} P'_k + P_k P'_{k-1})$. Similarly, we have $L M' = (P_kM_k + P_{k-1}M_{k+1})((-1)^k f^{-1}P'_k) = (-1)^k (P_kM_k f^{-1}P'_k + P_{k-1}ff^{-1}P'_k) =
(-1)^k (P_kM_k f^{-1}P'_k + P_{k-1}P'_k) = (-1)^k (P_k M_k f^{-1} P'_k + P_k P'_{k-1})$, where the last step follows from $P_{k-1}P'_k = P_k P'_{k-1}$. Thus $N'L_1 = LM'$ as required.

To establish invertibility, we also need to show that there are operators $A$ and $B$ with $M'M = 1 + AL$ and $MM' = 1 + BL_1$ as in equations~\eqref{inv:prop1} and~\eqref{inv:prop2}.

Similarly to the def\/initions of $P_i$ and $P'_i$, we def\/ine the sequences $R_i$ and $R'_i$ recursively. We take $R_1 = R'_1 = 1$, $R_2 = R'_2 = A_2$ and set $R_{i+1} = R_i A_{i+1} + R_{i-1}$ and $R'_{i+1} = A_{i+1} R'_i + R'_{i-1}$ for $i \geq 2$.

We can now prove $M_1 = R_i M_i + R_{i-1} M_{i+1}$ for $2 \leq i \leq k$, by induction. We have $M_1 = A_2 M_2 + M_3 = R_2 M_2 + R_1 M_3$ for our basis.
And assuming $M_1 = R_i M_i + R_{i-1} M_{i+1}$, we get $M_1 = R_i M_i + R_{i-1} M_{i+1} = R_i (A_{i+1} M_{i+1} + M_{i+2}) + R_{i-1} M_{i+1} =
(R_i A_{i+1} + R_{i-1}) M_{i+1} + R_i M_{i+2} = R_{i+1} M_{i+1} + R_i M_{i+2}$. This gives us $M = M_1 = R_k M_k + R_{k-1} M_{k+1}$.

Next we need several identities, and will prove by induction that $P'_i R_i = R'_i P_i$, $P'_i R_{i-1} = R'_i P_{i-1} + (-1)^i$ and $R'_{i-1} P_i = P'_{i-1} R_i + (-1)^i$ for $2 \leq i \leq k$. When $i = 2$, these become $(A_2 A_1 + 1)A_2 = A_2(A_1 A_2 + 1)$, $(A_2 A_1 + 1)1 = A_2 A_1 + (-1)^2$ and $1 (A_1 A_2 + 1) = A_1 A_2 + (-1)^2$, all of which are true. Now we assume $P'_i R_i = R'_i P_i$, $P'_i R_{i-1} = R'_i P_{i-1} + (-1)^i$ and $R'_{i-1} P_i = P'_{i-1} R_i + (-1)^i$, and calculate as follows. First, $P'_{i+1} R_{i+1} = (A_{i+1} P'_i + P'_{i-1})(R_i A_{i+1} + R_{i-1}) = A_{i+1} P'_i R_i A_{i+1} + A_{i+1} P'_i R_{i-1} + P'_{i-1} R_i A_{i+1} + P'_{i-1} R_{i-1}) = A_{i+1} R'_i P_i A_{i+1} + A_{i+1} (R'_i P_{i-1} + (-1)^i) + (R'_{i-1} P_i - (-1)^i) A_{i+1} + R'_{i-1} P_{i-1}) = A_{i+1} R'_i P_i A_{i+1} + A_{i+1} R'_i P_{i-1} + R'_{i-1} P_i A_{i+1} + R'_{i-1} P_{i-1}) = (A_{i+1} R'_i + R'_{i-1})(P_i A_{i+1} + P_{i-1}) =
R'_{i+1} P_{i+1}$, as desired. Next, $P'_{i+1} R_i = (A_{i+1} P'_i +P'_{i-1}) R_i = A_{i+1} P'_i R_i + P'_{i-1} R_i = A_{i+1} R'_i P_i + (R'_{i-1} P_i - (-1)^i) = (A_{i+1} R'_i + R'_{i-1}) P_i + (-1)^{i+1} = R'_{i+1} P_i + (-1)^{i+1}$, showing the second equality. Finally, $R'_i P_{i+1} =
R'_i (P_i A_{i+1} + P_{i-1}) = R'_i P_i A_{i+1} + R'_i P_{i-1} = P'_i R_i A_{i+1} + (P'_i R_{i-1} - (-i)^i) = P'_i (R_i A_{i+1} + R_{i-1}) + (-i)^{i+1}) =
P'_i R_{i+1} + (-i)^{i+1}$, giving the third equality. Thus we have $P'_k R_k = R'_k P_k$, $P'_k R_{k-1} = R'_k P_{k-1} + (-1)^k$ and $R'_{k-1} P_k = P'_{k-1} R_k + (-1)^k$.

We also need that $R_{k-1} P'_k = R_k P'_{k-1} + (-1)^k$. We leave the easier inductive proof of this to the reader.

Now we have $M'M = (-1)^k f^{-1} P'_k (R_k M_k + R_{k-1} M_{k+1}) = (-1)^k f^{-1} P'_k R_k M_k + (-1)^k f^{-1} P'_k $ $\times R_{k-1} M_{k+1} = (-1)^k f^{-1} R'_k P_k M_k + (-1)^k f^{-1} (R'_k P_{k-1} + (-1)^k) M_{k+1} = (-1)^k f^{-1} R'_k (P_k M_k + P_{k-1}) + (-1)^k f^{-1}(-1)^k M_{k+1} = f^{-1} M_{k+1} + (-1)^k f^{-1} R'_k (P_k M_k + P_{k-1}) = 1 + ((-1)^k f^{-1} R'_k)L$. This is $1 + AL$, where $A$ is $(-1)^k f^{-1} R'_k$.

Similarly, we get $MM' = (R_k M_k + R_{k-1} M_{k+1})(-1)^k f^{-1} P'_k = (-1)^k (R_k f^{-1} f M_k + R_{k-1} f^{-1} f$ $\times M_{k+1}) f^{-1} P'_k = (-1)^k (R_k f^{-1} N_k + R_{k-1} f^{-1} N_{k+1}) P'_k = (-1)^k (R_k f^{-1} N_k + R_{k-1}) P'_k$, since $N_k = f M_k f^{-1}$ and $N_{k+1} = f$. Continuing, $(-1)^k (R_k f^{-1} N_k + R_{k-1}) P'_k = (-1)^k (R_k f^{-1} N_k P'_k + R_{k-1} P'_k) = (-1)^k (R_k f^{-1} N_k P'_k + R_k P'_{k-1} + (-1)^k) = (-1)^k (R_k f^{-1} N_k P'_k + R_k f^{-1} N_{k+1} P'_{k-1}) + (-1)^2k = (-1)^k R_k f^{-1} (N_k P'_k + N_{k+1} P'_{k-1}) + 1 = (-1)^k R_k f^{-1} (N_k P'_k + N_{k+1} P'_{k-1}) + 1 = 1 + (-1)^k R_k f^{-1} L_1$. This is $1 + BL_1$, where $B = (-1)^k R_k f^{-1}$.
\end{proof}

\section{Continued Wronskian type Darboux transformations}\label{sec:weak:wr}

We may generalize the previous construction slightly, by removing the requirement that $M_{k+1}$ be a function. Much of the theory still works, provided $M_{k+1} = cB$ where $B$ and $M_k$ commute. We say that the resulting transformations are of {\em continued Wronskian type}. They generalize f\/irst-order Wronskian type transformations. Their relation with high-order Wronskians remains to be clarif\/ied. Here is the analogue of Theorem~\ref{Types_above_1_theorem}.

\begin{Theorem}\label{generalized iterated type theorem}
Let $k \geq 1$ be given, and suppose there are nonzero operators $L = M_0$, $A_1, A_2,$ $\dots, A_k$, $M = M_1, M_2, \dots, M_{k+1}$ so that
\begin{gather*}
M_{i-1} = A_i M_i + M_{i+1}
\end{gather*}
for $1 \leq i \leq k$, and there is a function $c$ with
\begin{gather*}
M_{k+1} = cB \notin K
\end{gather*}
and
\begin{gather*}
M_k B = B M_k .
\end{gather*}
Then we can define a Darboux transformation $NL = L_1M$ as follows. Define $N_{k+1} = M_{k+1}$ and $N_{k} = c M_k c^{-1}$, and define $N_i$ for $0 \leq i \leq k-1$ by downward recursion using $N_i = N_{i+1}A_{i+1} + N_{i+2}$. Finally let $N = N_1$, and let $L_1 = N_0$. Then $M$ and $N$ have the same principal
symbol, and $NL = L_1M$ is a Darboux transformation.
\end{Theorem}
\begin{proof}
The proof is essentially the same as that of Theorem \ref{Types_above_1_theorem}, so we will only indicate the parts that change. When $k = 1$ and $M_{k+1} \notin K$, we will not have a Type~I Darboux transformation.

The proof of $NL = L_1 M$ is by induction. We need the hypothesis that $M_k$ and $B$ commute for the basis, and have $N_{k+1}M_k = M_{k+1}M_k = cB M_k =
c M_k B = c M_k c^{-1} cB = N_kM_{k+1}$. After that, the proof works as before.
\end{proof}

Continued Wronskian type Darboux transformations are seldom invertible. As in our remarks preceding Theorem~\ref{iterated invertible theorem}, we have that
the Darboux transformation $(M,N) \colon L \rightarrow L_1$ is obtained by repeatedly forming shifts and duals, starting with $(M_k,N_k) \colon M_{k+1} \rightarrow N_{k+1}$ which is $(M_k,cM_kc^{-1}) \colon cB \rightarrow cB$. So $(M,N) \colon L \rightarrow L_1$ is invertible if\/f $(M_k,cM_kc^{-1}) \colon cB \rightarrow cB$ is, by Lemmas~\ref{shift invertibility lemma} and~\ref{dual is invertible lemma}.

Since a DT can only be invertible if it has $\ker L \cap \ker M = \{0\}$, it is worth noting that a~continued Wronskian type~DT has $\ker L \cap \ker M = \{0\}$ if\/f $\ker M_k \cap\ker M_{k+1} = \{ 0 \}$. In fact, $M_{i-1} = A_i M_i + M_{i+1}$ gives us $\ker M_{i-1} \cap\ker M_i = \ker M_i \cap\ker M_{i+1}$ for all $i$, so $\ker L \cap\ker M =\ker M_0 \cap\ker M_1 = \dots = \ker M_k \cap\ker M_{k+1}$.

\pdfbookmark[1]{References}{ref}
\LastPageEnding

\end{document}